\documentclass[a4paper,10pt]{amsart}

\usepackage{amsmath,amssymb,amsfonts}
\usepackage{hyperref}

\theoremstyle{plain}
\newtheorem{theorem}{Theorem}

\newtheorem{lemma}[theorem]{Lemma}
\newtheorem{corollary}[theorem]{Corollary}

\theoremstyle{definition}
\newtheorem{definition}[theorem]{Definition}
\newtheorem{remark}[theorem]{Remark}

\DeclareMathOperator*{\argmin}{arg\,min}
\DeclareMathOperator{\range}{Ran}
\DeclareMathOperator{\domain}{dom}
\newcommand{\field}[1]{\mathbb{#1}}
\newcommand{\R}{\field{R}}
\newcommand{\N}{\field{N}}

\DeclareMathOperator{\D}{\mathcal{D}}

\begin{document}

\title{Source Conditions for non-quadratic Tikhonov Regularisation}
\author{Markus Grasmair}
\address{Department of Mathematical Sciences \\ Norwegian University of Science and Technology \\ N-7491 Trondheim, Norway}
\email{\href{mailto:markus.grasmair@ntnu.no}{markus.grasmair@ntnu.no}}
\date{\today}
\subjclass[2010]{47A52; 49N45; 65J20}
\keywords{Linear inverse problems; Tikhonov regularisation; convergence rates; source conditions}

\begin{abstract}
In this paper we consider convex Tikhonov regularisation for the solution
of linear operator equations on Hilbert spaces.
We show that standard fractional source conditions can be employed in
order to derive convergence rates in terms of the Bregman distance,
assuming some stronger convexity properties of either the
regularisation term or its convex conjugate.
In the special case of quadratic regularisation, we are able to reproduce
the whole range of H\"older type convergence rates known from classical theory.
\end{abstract}

\maketitle

\section{Introduction}

In the recent years, considerable progress has been made concerning the
analysis of convex Tikhonov regularisation in various settings.
Existence, stability, and convergence have been treated exhaustively
in different settings including that of non-linear problems in Banach spaces
with different similarity and regularisation terms.
Moreover, starting with the paper~\cite{BurOsh04}, the questions
of reconstruction accuracy and asymptotic error estimates have
gradually been answered.

The setting of~\cite{BurOsh04}, which we will also pursue in this paper,
is that of the stable solution of a linear, but noisy and ill-posed, operator
equation
\[
Fu = v^\delta
\]
by means of Tikhonov regularisation
\begin{equation}\label{eq:func}
u_\alpha^\delta = \argmin_u \Bigl(\frac{1}{2}\lVert Fu - v^\delta\rVert^2 + \alpha\mathcal{R}(u)\Bigr),
\end{equation}
with a quadratic similarity term but the convex and lower semi-continuous regularisation
term $\mathcal{R}$.
It was shown in~\cite{BurOsh04} that the source condition
\[
\xi^\dagger = F^*\omega^\dagger \in \partial\mathcal{R}(u^\dagger),
\]
with $u^\dagger$ being the solution of the noise-free equation, implies the
error estimate
\[
\D_{\xi^\dagger}(u_\alpha^\delta,u^\dagger) \lesssim \delta
\]
for a parameter choice $\alpha \sim \delta$.
Here $\D_{\xi^\dagger}$ denotes the Bregman distance for the functional $\mathcal{R}$,
which is defined as
\[
\D_{\xi^\dagger}(u_\alpha^\delta,u^\dagger) 
= \mathcal{R}(u_\alpha^\delta) - \mathcal{R}(u^\dagger) - \langle \xi^\dagger, u_\alpha^\delta - u^\dagger\rangle.
\]
This result can be seen as a direct generalisation of the classical result
for quadratic regularisation with $\mathcal{R}(u) = \frac{1}{2}\lVert u \rVert^2$,
where we have the convergence rate
\[
\lVert u_\alpha^\delta - u^\dagger \rVert \lesssim \delta^{1/2}
\qquad\text{ if }\qquad
u^\dagger = F^*\omega^\dagger,
\]
again for the parameter choice $\alpha \sim \delta$.
This is due to the fact that the sub-differential of the 
regularisation term consists in this case of the single element $u^\dagger$,
and the Bregman distance is simply the squared norm of the difference of the arguments.
The classical results, however, are in fact significantly more general,
as they can be easily extended to fractional source conditions
leading to rates of the form
\[
\lVert u_\alpha^\delta - u^\dagger\rVert \lesssim \delta^{\frac{2\nu}{2\nu+1}}
\]
if the source condition
\begin{equation}\label{eq:classical}
u^\dagger = (F^*F)^\nu \omega^\dagger
\end{equation}
holds for some $0 < \nu \le 1$ and the regularisation parameter $\alpha$ is chosen appropriately.

In order to generalise these results to non-linear operators $F$,
the paper~\cite{Poe} introduced the idea of variational inequalities,
which were later modified in~\cite{BotHof10,Gra10}
in order to deal with lower regularity of the solution as well.
As alternative, the idea of approximate source conditions was introduced
first for quadratic regularisation \cite{Hof06} and then generalised to
non-quadratic situations~\cite{Hei08}.
In their original form, both of these approaches dealt,
in the non-quadratic case,
only with lower order convergence rates;
in the quadratic setting, this would roughly correspond to
the classical source condition~\eqref{eq:classical} with $\nu \le 1/2$.
However, modifications were proposed for approximate source
conditions in~\cite{Hei09,Neu} and for variational inequalities in~\cite{Gra13} in order to
accommodate for a higher regularity as well,
roughly corresponding to~\eqref{eq:classical} with $1/2 < \nu \le 1$.

In contrast to the relatively simple source condition~\eqref{eq:classical},
variational inequalities and approximate source conditions can be hard to interpret and
verify in concrete settings.
Thus it would be desirable to obtain restatements
in terms of more palpable conditions and to clarify the relation
between the different variational and approximate conditions and 
standard source conditions.
For the quadratic case, this relation has been made clear in~\cite{Elb}.
For the non-quadratic case, however, such an analysis is, as of now,
not available.

\subsection{Summary of results}

In this article, we will consider convex Tikhonov regularisation
for linear inverse problems on Hilbert spaces of the form~\eqref{eq:func}.
The goal of this article is the derivation of convergence rates,
that is, estimates for the difference between the reconstruction
$u_\alpha^\delta$ and the true solution $u^\dagger$
under the natural generalisation
\begin{equation}\label{eq:source}
\xi^\dagger = (F^*F)^\nu \omega^\dagger \in \partial\mathcal{R}(u^\dagger)
\end{equation}
of the classical source condition~\eqref{eq:func} to convex regularisation
terms.
The following theorem briefly summarises the main results obtained in this paper,
see Theorems~\ref{th:rates1}, \ref{th:pconv}, and~\ref{th:higher}.
For an overview of the notation used here, see Section~\ref{se:prel}.

\begin{theorem}
Assume that a source condition of the form~\eqref{eq:source}
holds for some $0 < \nu \le 1$. Then we have the following convergence rates:
\begin{itemize}
\item For $0 < \nu \le 1/2$ we have
\[
\D_{\xi^\dagger}(u_\alpha^\delta,u^\dagger) \lesssim \delta^{2\nu}
\qquad\text{ for }\qquad
\alpha \sim \delta^{2-2\nu}.
\]
\item If $\mathcal{R}$ is $p$-convex (see Definition~\ref{de:pconvex}) and
$0 < \nu \le 1/2$ we have
\[
\D_{\xi^\dagger}(u_\alpha^\delta,u^\dagger) \lesssim \delta^{\frac{2\nu p}{p-1+2\nu}}
\qquad\text{ for }\qquad
\alpha \sim \delta^{\frac{2p-2-2p\nu+4\nu}{p-1+\nu}}.
\]
\item If $\mathcal{R}$ is $q$-coconvex (see Definition~\ref{de:qcoconvex}) and
$1/2 \le \nu \le 1$ we have
\[
\D_{\xi_\alpha^\delta,\xi^\dagger}^{{\rm sym}}(u_\alpha^\delta,u^\dagger)
\lesssim \delta^{\frac{2\nu q}{1+2\nu q - 2\nu}}
\qquad\text{ for }\qquad
\alpha \sim \delta^{\frac{2+2\nu q - 4\nu}{1+2\nu q - 2\nu}}.
\]
\end{itemize}
\end{theorem}

In the case of quadratic regularisation with $\mathcal{R}(u) = \frac{1}{2}\lVert u \rVert^2$,
all of these results coincide with the classical results found, for instance, in~\cite{Gro84}.
In Section~\ref{se:examples}, we will in addition discuss the implications for
several examples of non-quadratic regularisation terms.

\section{Mathematical preliminaries}\label{se:prel}

Let $U$ and $V$ be Hilbert spaces and $F \colon U \to V$ a bounded linear operator.
Moreover, let $\mathcal{R}\colon U \to [0,+\infty]$ be a convex, lower semi-continuous
and coercive functional. Given some data $v \in V$, we consider the stable, approximate
solution of the equation $Fu = v$ by means of non-quadratic Tikhonov regularisation,
that is, by minimising the functional
\[
\mathcal{T}_\alpha(u,v) := \frac{1}{2}\lVert Fu-v\rVert^2 + \alpha\mathcal{R}(u).
\]

More precisely, we assume that $v^\dagger\in V$ is some "true" data,
but that we are only given noisy data $v^\delta \in V$ satisfying
\[
\lVert v^\dagger - v^\delta\rVert \le \delta
\]
for some noise level $\delta > 0$.
Moreover, we denote the true, that is, $\mathcal{R}$-minimising, solution of
the noise-free equation $Fu=v^\dagger$ by
\[
u^\dagger :\in \argmin_u\bigl\{\mathcal{R}(u) : Fu = v^\dagger\bigr\}.
\]
Our main goal is the estimation of the worst case reconstruction error 
\[
\sup\bigl\{D(u_\alpha^\delta,u^\dagger) : u_\alpha^\delta \in \argmin_u\mathcal{T}_\alpha(u,v^\delta),\ \lVert v^\dagger-v^\delta\rVert \le \delta\bigr\}.
\]
Here, $D \colon U \times U \to [0,+\infty]$ is some distance like measure.
In the following results we will mostly use the Bregman distance with respect to the
regularisation functional $\mathcal{R}$, which is defined as
\[
\D_\xi(\tilde{u},u):= \mathcal{R}(\tilde{u}) - \mathcal{R}(u) - \langle \xi, \tilde{u}-u\rangle,
\]
where 
\[
\xi \in \partial\mathcal{R}(u)
\]
is some sub-gradient of $\mathcal{R}$ at $u$.
In addition, we will consider the symmetric Bregman distance
\[
\D^{{\rm sym}}_{\xi,\tilde{\xi}} := \D_\xi(\tilde{u},u) + \D_{\tilde{\xi}}(u,\tilde{u})
= \langle \xi - \tilde{\xi}, u-\tilde{u}\rangle
\]
for
\[
\xi \in \partial\mathcal{R}(u)
\qquad\text{ and }\qquad
\tilde{\xi} \in \partial\mathcal{R}(\tilde{u}),
\]
as well as the norm in some instances.

\subsection{Existence, convergence, and stability}

It is well known that Tikhonov regularisation
with a convex, lower semi-continuous, and coercive regularisation term
is a well-defined regularisation method.
That is, the following results hold (see~\cite[Thms.~3.22,~3.23,~3.26]{VMII}):

\begin{itemize}
\item For every $v \in V$ and every $\alpha > 0$, the functional $\mathcal{T}_\alpha(\cdot,v)$
attains its minimum.
\item Assume that $v_k \to v \in V$ and $\alpha_k \to \alpha > 0$,
and let $u_k \in \argmin_u \mathcal{T}_{\alpha_k}(u,v_k)$. Then the sequence
$u_k$ has a weakly convergent sub-sequence. Moreover, if $\bar{u}$ is the
weak limit of any weakly convergent sub-sequence $(u_{k'})$, then
\[
\bar{u} \in \argmin_u \mathcal{T}_\alpha(u,v)
\qquad\text{ and }\qquad
\mathcal{R}(u_{k'}) \to \mathcal{R}(\bar{u}).
\]
\item Assume that 
\begin{equation}\label{eq:ratio}
\delta_k \to 0,\qquad \alpha_k \to 0,\qquad\text{ and }\qquad\delta_k^2/\alpha_k \to 0.
\end{equation}
Let moreover $v_k \in V$ satisfy $\lVert v_k - v^\dagger\rVert \le \delta_k$, and let
$u_k \in \argmin_u\mathcal{T}_{\alpha_k}(u,v_k)$.
Then the sequence $u_k$ has a sub-sequence $(u_{k'})$ that converges weakly
to some $\mathcal{R}$-minimising solution $\bar{u}$ of the equation $Fu = v^\dagger$
and $\mathcal{R}(u_{k'}) \to \mathcal{R}(\bar{u})$.
\end{itemize}

\begin{remark}
If the functional $\mathcal{T}_\alpha(\cdot,v)$ is strictly convex,
which is the case, if and only if the restriction of $\mathcal{R}$ to the kernel of $F$
is strictly convex, then the minimiser of $\mathcal{T}_\alpha(\cdot,v)$ as well
as the $\mathcal{R}$-minimising solution of $Fu=v^\dagger$ are unique.
In such a case, a standard sub-sequence argument shows that the whole
sequences $u_k$ converge weakly to $\bar{u}$.
\end{remark}

\begin{remark}\label{re:bd}
The fact that $u_\alpha^\delta$ minimises the Tikhonov functional $\mathcal{T}_\alpha(\cdot,v^\delta)$
implies that
\begin{equation}\label{eq:bdh1}
\frac{1}{2}\lVert Fu_\alpha^\delta - v^\delta\rVert^2 + \alpha\mathcal{R}(u_\alpha^\delta)
\le \frac{1}{2}\lVert Fu^\dagger - v^\delta\rVert^2 + \alpha\mathcal{R}(u^\dagger)
\le \frac{\delta^2}{2} + \alpha\mathcal{R}(u^\dagger),
\end{equation}
which in turn implies in particular that
\[
\mathcal{R}(u_\alpha^\delta) \le \frac{\delta^2}{2\alpha} + \mathcal{R}(u^\dagger).
\]
Because of the coercivity of $\mathcal{R}$, it follows that there exists some
constant $R = R(\delta^2/\alpha,u^\dagger)$ only depending on the ratio $\delta^2/\alpha$
and the true solution $u^\dagger$ (or, rather, the function value $\mathcal{R}(u^\dagger)$
at the true solution) such that
\begin{equation}\label{eq:solbound}
\lVert u_\alpha^\delta \rVert \le R(\delta^2/\alpha,u^\dagger).
\end{equation}
We will in the following always be interested in the case where $u^\dagger$
is a fixed $\mathcal{R}$-minimising solution of $Fu = v^\dagger$
and the noise level $\delta$ is small and thus, due to the requirement~\eqref{eq:ratio}
on the regularisation parameter,
also the ratio $\delta^2/\alpha$.
Therefore, we can always assume that all the regularised solutions $u_\alpha^\delta$
are uniformly bounded.
\end{remark}

\begin{remark}\label{re:coercive}
Throughout this paper, we assume that the regularisation term $\mathcal{R}$
is coercive, as this guarantees the well-posedness of the regularisation method
as well as the bound~\eqref{eq:solbound}, which is needed for the derivation
of the convergence rates later on.
However, both of these can also be guaranteed under the weaker condition
that the Tikhonov functional $\mathcal{T}_\alpha(\cdot,v)$ is coercive
for any or, equivalently, every $\alpha > 0$ and $v \in V$.
For the well-posedness see again~\cite[Thms.~3.22,~3.23,~3.26]{VMII};
the bound follows from the inequality (cf.~\eqref{eq:bdh1})
\[
\frac{1}{2}\lVert Fu_\alpha^\delta - v\rVert^2
\le \lVert Fu_\alpha^\delta - v^\delta \rVert^2
+ \lVert v-v^\delta\rVert^2
\le \delta^2 + 2\alpha\mathcal{R}(u^\dagger)+\lVert v - v^\delta\rVert^2
\]
and the fact that $v^\delta \to v^\dagger$ implying that $\lVert v - v^\delta \rVert$
remains bounded for every fixed $v \in V$.
Thus all the results of this paper remain valid under this more general coercivity condition.

In particular, this holds for regularisation with (higher order)
homogeneous Sobolev norms or (higher order) total variation
\[
\mathcal{R}(u) = \lVert \nabla^\ell (u) \rVert_{L^p}^p
\qquad\text{ or }\qquad
\mathcal{R}(u) = \lvert D^\ell (u) \rvert(\Omega)
\]
with $\ell \in \N$ and $1 < p < +\infty$ provided that the domain
$\Omega$ is connected and the kernel of $F$ does not contain any
polynomials of degree at most $\ell-1$. See for instance
\cite{AcaVog94,Ves01} for the total variation case,
\cite[Prop.~3.66, 3.70]{VMII} for quadratic Sobolev and total variation regularisation,
and~\cite{Gra11} for the general, abstract case.
\end{remark}

\subsection{An interpolation inequality}

All of the convergence rate results in this paper are based at some
point on the following interpolation inequality,
which can, for instance be found in~\cite[p.~47]{EHN}:

\begin{lemma}
For all $0\le\nu\le 1/2$ and all $u \in U$ we have
\begin{equation}\label{eq:interpol_orig}
\lVert (F^*F)^\nu u\rVert
\le \lVert Fu\rVert^{2\nu}\lVert u\rVert^{1-2\nu}.
\end{equation}
\end{lemma}

More precisely, we will make use of the following result:

\begin{corollary}
Let $0 \le \nu \le 1/2$ and assume that $\xi \in U$ satisfies
\[
\xi = (F^*F)^\nu \omega
\]
for some $\omega \in U$.
Then
\begin{equation}\label{eq:interpol}
\langle \xi,u\rangle \le \lVert \omega \rVert \lVert Fu\rVert^{2\nu} \lVert u \rVert^{1-2\nu}
\end{equation}
for all $u \in U$.
\end{corollary}

\begin{proof}
With the interpolation inequality~\eqref{eq:interpol_orig} we have
\[
\langle\xi,u\rangle
= \langle(F^*F)^\nu\omega,u\rangle
= \langle\omega,(F^*F)^\nu u\rangle
\le \lVert \omega\rVert\lVert(F^*F)^\nu u\rVert
\le \lVert\omega\rVert\lVert Fu\rVert^{2\nu} \lVert u \rVert^{1-2\nu},
\]
which proves the assertion.
\end{proof}

\section{Basic convergence rates}

We consider first the case of a lower order fractional source condition
of the form 
\[
\xi^\dagger \in \range (F^*F)^{\nu} \cap \partial\mathcal{R}(u^\dagger)
\]
with $0 < \nu \le 1/2$ without any additional
conditions on the regularisation term $\mathcal{R}$. 
The limiting case $\nu = 1/2$ can be equivalently written as the more
standard source condition $\xi^\dagger \in \range(F^*) \cap \partial\mathcal{R}(u^\dagger)$,
for which it is well known that one obtains a convergence rate
\[
\D_{\xi^\dagger}(u_\alpha^\delta,u^\dagger) \lesssim \delta
\qquad\text{ for }\qquad \alpha \sim \delta.
\]
The following result shows that a weaker source condition leads
to a correspondingly slower convergence.

\begin{theorem}\label{th:rates1}
Assume that there exists
\[
\xi^\dagger := (F^*F)^{\nu}\omega^\dagger \in \partial\mathcal{R}(u^\dagger)
\]
for some $0 < \nu \le 1/2$.
Then
\[
\D_{\xi^\dagger}(u_\alpha^\delta,u^\dagger)
\lesssim C_1\frac{\delta^2}{\alpha} + C_2\delta^{2\nu} + C_3\alpha^{\frac{\nu}{1-\nu}}.
\]
for some constants $C_1$, $C_2$, $C_3 > 0$ whenever $\delta^2/\alpha$ is uniformly bounded.
In particular, one obtains with a parameter choice 
\[
\alpha(\delta) \sim \delta^{2-2\nu}
\]
a convergence rate 
\[
\D_{\xi^\dagger}(u_\alpha^\delta,u^\dagger) \lesssim \delta^{2\nu}.
\]
\end{theorem}

\begin{proof}
We will only consider the case $0 < \nu < 1/2$,
the case $\nu = 1/2$ having already been treated in~\cite{BurOsh04}.

Since $\xi^\dagger = (F^*F)^\nu\omega^\dagger$, we can apply the 
interpolation inequality~\eqref{eq:interpol}, which yields that
\[
\langle\xi^\dagger,u^\dagger - u\rangle
\le \lVert\omega^\dagger\rVert\lVert F(u^\dagger-u)\rVert^{2\nu}\lVert u-u^\dagger\rVert^{1-2\nu}.
\]
Moreover, the fact that
$u_\alpha^\delta$ minimises the Tikhonov functional implies that
\[
\frac{1}{2}\lVert Fu_\alpha^\delta-v^\delta\rVert^2 + \alpha\mathcal{R}(u_\alpha^\delta)
\le \frac{1}{2}\lVert Fu^\dagger-v^\delta\rVert^2 + \alpha\mathcal{R}(u^\dagger)
\le \frac{1}{2}\delta^2 + \alpha\mathcal{R}(u^\dagger).
\]
Thus
\begin{multline}\label{eq:rateh1}
\D_{\xi^\dagger}(u_\alpha^\delta;u^\dagger)
= \mathcal{R}(u_\alpha^\delta) - \mathcal{R}(u^\dagger) - \langle \xi^\dagger,u_\alpha^\delta-u^\dagger\rangle\\
\le \frac{\delta^2}{2\alpha} - \frac{1}{2\alpha}\lVert Fu_\alpha^\delta-v^\delta\rVert^2
+ \lVert \omega \rVert\lVert F(u^\dagger-u_\alpha^\delta)\rVert^{2\nu}\lVert u^\dagger-u_\alpha^\delta\rVert^{1-2\nu}.
\end{multline}
Using Remark~\ref{re:bd} we see that the term $\lVert u^\dagger-u_\alpha^\delta\rVert$ stays bounded. Using
the fact that
\[
\lVert F(u^\dagger-u_\alpha^\delta)\rVert^{2\nu}
\le \lVert Fu_\alpha^\delta-v^\delta\rVert^{2\nu} + \delta^{2\nu},
\]
we obtain thus from~\eqref{eq:rateh1} the estimate
\[
\D_{\xi^\dagger}(u_\alpha^\delta;u^\dagger)
\le \frac{\delta^2}{2\alpha} + C\delta^{2\nu} - \frac{1}{2\alpha}\lVert Fu_\alpha^\delta-v^\delta\rVert^2
+ C\lVert Fu_\alpha^\delta-v^\delta\rVert^{2\nu}
\]
for some $C > 0$.
Using Young's inequality $ab \le a^p/p + b^{p_*}/{p_*}$,
we see that
\[
C\lVert Fu_\alpha^\delta-v^\delta\rVert^{2\nu} \le \frac{1}{2\alpha}\lVert Fu_\alpha^\delta-v^\delta\rVert^2
+ \tilde{C}\alpha^{\frac{\nu}{1-\nu}}
\]
for some $\tilde{C} > 0$,
and thus
\[
\D_{\xi^\dagger}(u_\alpha^\delta;u^\dagger)
\le \frac{\delta^2}{2\alpha} + C\delta^{2\nu} + \tilde{C}\alpha^{\frac{\nu}{1-\nu}}.
\]
Now the rate follows immediately by inserting the parameter choice 
$\alpha \sim \delta^{2-2\nu}$.
\end{proof}

\begin{remark}\label{re:quad_low}
In quadratic Tikhonov regularisation with
\[
\mathcal{R}(u) = \frac{1}{2}\lVert u \rVert^2
\]
we have that
\[
\partial\mathcal{R}(u^\dagger) = u^\dagger
\qquad\text{ and }\qquad
\D_{u^\dagger}(u,u^\dagger) = \frac{1}{2}\lVert u -u^\dagger \rVert^2.
\]
Thus the condition of Theorem~\ref{th:rates1} reduces to the classical
(lower order) source condition
\[
u^\dagger \in \range(F^*F)^\nu
\qquad\text{ with }\qquad 0 < \nu \le 1/2.
\]
The convergence rate obtained in Theorem~\ref{th:rates1}, however, would be
\[
\lVert u_\alpha^\delta - u^\dagger \rVert \lesssim \delta^\nu
\qquad\text{ with }\qquad
\alpha \sim \delta^{2-2\nu}.
\]
In contrast, it is well known (see e.g~\cite{Gro84}) that a parameter choice
\[
\alpha \sim \delta^{\frac{2}{2\nu + 1}}
\]
leads to a convergence rate
\[
\lVert u_\alpha^\delta - u^\dagger \rVert \lesssim \delta^{\frac{2\nu}{2\nu + 1}}.
\]
Since $\nu > 2\nu/(2\nu+1)$ for $0 < \nu < 1/2$, this convergence rate
is faster than the one obtained in the Theorem~\ref{th:rates1}.
The reason for this discrepancy can be found in the inequality~\eqref{eq:rateh1},
after which we estimate the term $\lVert u^\dagger-u_\alpha^\delta\rVert$ simply
by a constant. Here better estimates are possible, if we can use some power
of the Bregman distance in order to bound this term from above.
For quadratic regularisation, this is obviously possible, as the Bregman
distance is essentially the squared norm.
More general instances of this situation will be discussed in the following section.
\end{remark}

\section{Convergence rates for $p$-convex functionals}

As discussed above, in order to obtain stronger results, we need to require a stronger form of convexity
for the regularisation term $\mathcal{R}$.

\begin{definition}\label{de:pconvex}
Let $1 \le p < +\infty$. We say that the functional $\mathcal{R}\colon U \to [0,+\infty]$
is locally $p$-convex, if there exists for each $u \in \domain \partial\mathcal{R}$
and every $R > 0$ some constant $C = C(u,R) > 0$ such that
\[
C \lVert \tilde{u}-u\rVert^p \le \D_{\xi}(\tilde{u},u)
\]
for all $\xi \in \partial\mathcal{R}(u)$ and all $\tilde{u} \in U$
with $\lVert \tilde{u} - u \rVert \le R$.
\end{definition}

\begin{theorem}\label{th:pconv}
Assume that $\mathcal{R}$ is locally $p$-convex for some $p \ge 1$ and that there exists
\[
\xi^\dagger := (F^*F)^\nu \omega^\dagger \in \partial\mathcal{R}(u^\dagger)
\]
for some $0 < \nu < 1/2$.
Then there exist constants $C_1$, $C_2 > 0$ such that
\[
\D_{\xi^\dagger}(u_\alpha^\delta,u^\dagger) \le C_1 \frac{\delta^2}{\alpha} + C_2 \alpha^{\frac{\nu p}{p-1-p\nu+2\nu}}
\]
whenever $\delta^2/\alpha$ is uniformly bounded.
In particular, we obtain with a parameter choice
\[
\alpha(\delta) \sim \delta^{\frac{2p-2-2p\nu+4\nu}{p-1+\nu}}
\]
the convergence rate
\[
\D_{\xi^\dagger}(u_\alpha^\delta,u^\dagger) \lesssim \delta^{\frac{2\nu p}{p-1+2\nu}}.
\]
\end{theorem}

\begin{proof}
As in the proof of Theorem~\ref{th:rates1} we obtain the estimate (cf.~inequality~\eqref{eq:rateh1})
\[
\D_{\xi^\dagger}(u_\alpha^\delta;u^\dagger)
\le \frac{\delta^2}{2\alpha} - \frac{1}{2\alpha}\lVert Fu_\alpha^\delta-v^\delta\rVert^2
+ \lVert\omega\rVert\lVert F(u^\dagger-u_\alpha^\delta)\rVert^{2\nu}\lVert u^\dagger-u_\alpha^\delta\rVert^{1-2\nu}.
\]
Again, it follows from Remark~\ref{re:bd} that we can assume the term $\lVert u^\dagger - u_\alpha^\delta \rVert$ 
to be bounded. Thus the local $p$-convexity of $\mathcal{R}$ implies the existence of a constant $C$
such that
\[
\lVert u^\dagger-u_\alpha^\delta\rVert \le C\D_{\xi^\dagger}(u_\alpha^\delta,u^\dagger)^{\frac{1}{p}}
\]
and we obtain the estimate
\begin{equation}\label{eq:h2_a}
\D_{\xi^\dagger}(u_\alpha^\delta;u^\dagger)
\le \frac{\delta^2}{2\alpha} - \frac{1}{2\alpha}\lVert Fu_\alpha^\delta-v^\delta\rVert^2
+ C^{1-2\nu}\lVert \omega\rVert\lVert F(u^\dagger-u_\alpha^\delta)\rVert^{2\nu}\D_{\xi^\dagger}(u_\alpha^\delta,u^\dagger)^{\frac{1-2\nu}{p}}.
\end{equation}
We now apply Young's inequality
\begin{equation}\label{eq:Y3}
abc \le \frac{1}{r}a^r + \frac{1}{s}b^s + \frac{1}{t}c^t
\text{ for } a,\,b,\,c > 0 \text{ and } r,\,s,\,t > 1 \text{ with }
\frac{1}{r}+\frac{1}{s}+\frac{1}{t}=1
\end{equation}
with
\[
\begin{aligned}
a &= C^{1-2\nu}\frac{(4\alpha)^\nu \lVert \omega^\dagger\rVert}{\nu^\nu} , &&& r &= \frac{p}{p-1-p\nu + 2\nu},\\
b &= \frac{\nu^\nu}{(4\alpha)^\nu}\lVert F(u^\dagger-u_\alpha^\delta)\rVert^{2\nu}, &&& s &= \frac{1}{\nu},\\
c &= D_{\xi^\dagger}(u_\alpha^\delta,u^\dagger)^{\frac{1-2\nu}{p}}, &&& t &= \frac{p}{1-2\nu},
\end{aligned}
\]
which results in the bound
\begin{equation}\label{eq:h2_b}
\lVert\omega^\dagger\rVert \lVert F(u^\dagger-u_\alpha^\delta)\rVert
\le \tilde{C} \alpha^{\frac{\nu p}{p-1-p\nu+2\nu}} + \frac{1}{4\alpha} \lVert F(u^\dagger-u_\alpha^\delta)\rVert^2
+ \frac{1-2\nu}{p} \D_{\xi^\dagger}(u_\alpha^\delta,u^\dagger)
\end{equation}
for some constant $\tilde{C} > 0$.
Using that
\[
\lVert F(u^\dagger-u_\alpha^\delta)\rVert^2 \le 2 \lVert Fu_\alpha^\delta - v^\delta \rVert^2 + 2 \lVert Fu^\dagger - v^\delta\rVert^2
\le 2\lVert Fu_\alpha^\delta - v^\delta\rVert^2 + 2\delta^2,
\]
and combining~\eqref{eq:h2_a} with~\eqref{eq:h2_b}, we obtain the required inequality
\[
\D_{\xi^\dagger}(u_\alpha^\delta,u^\dagger)
\le C_1 \frac{\delta^2}{\alpha} + C_2 \alpha^{\frac{\nu p}{p-1-p\nu+2\nu}}
\]
for some $C_1$, $C_2$, $C_3 > 0$.
The two terms on the right hand side of this estimate balance for
\[
\alpha \sim \delta^{\frac{2p-2-2p\nu+4\nu}{p-1+2\nu}},
\]
in which case we obtain the convergence rate
\[
\D_{\xi^\dagger}(u_\alpha^\delta,u^\dagger) \lesssim \delta^{\frac{2\nu p}{p-1+2\nu}}.
\]
\end{proof}

\begin{remark}\label{re:quad_low2}
Assume that the assumptions of Theorem~\ref{th:pconv} are satisfied.
Because of the local $p$-convexity of $\mathcal{R}$,
we then obtain in addition a convergence rate in terms of the norm of the form
\[
\lVert u_\alpha^\delta - u^\dagger\rVert \lesssim \delta^{\frac{2\nu}{p-1+2\nu}}.
\]
In the particular case of a $2$-convex regularisation term, we recover
the familiar convergence rate
\[
\lVert u_\alpha^\delta - u^\dagger\rVert \lesssim \delta^{\frac{2\nu}{1+2\nu}}
\qquad\text{ for }\qquad
\xi^\dagger \in \range(F^*F)^\nu,\ 0 < \nu \le 1/2,
\]
with a parameter choice
\[
\alpha \sim \delta^{\frac{2}{1+2\nu}},
\]
which is the same as we obtain for quadratic Tikhonov regularisation
(cf.~Remarks~\ref{re:quad_low}).
\end{remark}

\section{Higher order rates}

We will now consider higher order source conditions
\[
\xi^\dagger \in \range(F^*F)^\nu \cap \partial\mathcal{R}(u^\dagger)
\qquad\text{ with }\qquad \frac{1}{2} < \nu \le 1.
\]
Here it turns out that a strong type of convexity appears not to be needed
to obtain higher order convergence rates. Instead, it is the convexity of
the conjugate of the regularisation term $\mathcal{R}$ that needs to be controlled.

\begin{definition}\label{de:qcoconvex}
Let $1 \le q < +\infty$.
We say that the functional $\mathcal{R}\colon U \to [0,+\infty]$
is locally $q$-coconvex, if there exists for all $R > 0$
some constant $C = C(R) > 0$ such that
\[
C\lVert \xi_1-\xi_2\rVert^q \le \D_{\xi_1,\xi_2}^{{\rm sym}}(u_1,u_2)
= \langle \xi_1 - \xi_2, u_1 - u_2 \rangle
\]
for all $u_1$, $u_2 \in \domain\partial\mathcal{R}$ with $\lVert u_i \rVert \le R$,
where
\[
\xi_1 \in \partial\mathcal{R}(u_1)
\qquad\text{ and }\qquad
\xi_2 \in \partial\mathcal{R}(u_2).
\]
\end{definition}

\begin{remark}
Instead of the original functional $\mathcal{R}$, we can also
consider its convex conjugate $\mathcal{R}^*$ and the dual
Bregman distances
\[
\D^*_u(\tilde{\xi},\xi) 
= \mathcal{R}^*(\tilde{\xi}) - \mathcal{R}^*(\xi) - \langle u,\tilde{\xi}-\xi\rangle
\qquad\text{ with }\qquad u \in \partial\mathcal{R}^*(\xi)
\]
and
\[
\D^{{\rm sym},*}_{u,\tilde{u}}(\xi,\tilde{\xi}) := \D^*_u(\tilde{\xi},\xi) + \D^*_{\tilde{u}}(\xi,\tilde{\xi})
\quad\text{ with } u \in \partial\mathcal{R}^*(\xi)
\text{ and } \tilde{u} \in \partial\mathcal{R}^*(\tilde{\xi}).
\]
Then we see that the primal and dual symmetric Bregman distances
are identical in the sense that
\[
\D_{\xi,\tilde{\xi}}^{{\rm sym}}(u,\tilde{u})
= \langle \xi-\tilde{\xi},u-\tilde{u}\rangle
= \D^{{\rm sym},*}_{u,\tilde{u}}(\xi,\tilde{\xi}).
\]
As a consequence, the $q$-coconvexity of $\mathcal{R}$
is equivalent to the $q$-convexity of $\mathcal{R}^*$.
Also, we note that $2$-coconvexity of $\mathcal{R}$ is the same as
cocoercivity of the subgradient $\partial\mathcal{R}$
(cf.~\cite[Sec.~4.2]{BauCom}).
\end{remark}

\begin{theorem}\label{th:higher}
Assume that $\mathcal{R}$ is locally $q$-coconvex for some $q \ge 1$ and that
\[
\xi^\dagger := (F^*F)^\nu \eta^\dagger \in \partial\mathcal{R}(u^\dagger)
\]
for some $1/2 < \nu \le 1$. Then there exist constants $C_1$, $C_2 > 0$ such that
\[
\D_{\xi_\alpha^\delta,\xi^\dagger}^{{\rm sym}}(u_\alpha^\delta,u^\dagger)
\le C_1\frac{\delta^2}{\alpha} + C_2 \alpha^{\frac{q\nu}{1+\nu q -2\nu}}.
\]
whenever $\delta^2/\alpha$ is uniformly bounded.
In particular, we obtain with a parameter choice
\[
\alpha \sim \delta^{\frac{2+2\nu q - 4\nu}{1+2\nu q - 2\nu}}
\]
the convergence rate
\[
\D_{\xi_\alpha^\delta,\xi^\dagger}^{{\rm sym}}(u_\alpha^\delta,u^\dagger)
\lesssim \delta^{\frac{2\nu q}{1+2\nu q - 2\nu}}.
\]
\end{theorem}

\begin{proof}
Denote
\[
\mu = \nu - \frac{1}{2}.
\]
Since
\[
\range (F^* F)^\nu
=  \range (F^*F)^{\mu + \frac{1}{2}}
= \range (F^* (FF^*)^\mu),
\]
it follows that we can write
\[
\xi^\dagger = F^*\omega^\dagger
\qquad\text{ with }\qquad
\omega^\dagger = (FF^*)^\mu \tilde{\eta}^\dagger
\]
for some $\tilde{\eta}^\dagger \in U$.

Because $u_\alpha^\delta$ is a minimiser of the Tikhonov functional
$\mathcal{T}_\alpha(\cdot,v^\delta)$, it satisfies the first order optimality
condition
\[
F^*(Fu_\alpha^\delta - v^\delta) + \alpha\partial\mathcal{R}(u_\alpha^\delta) \ni 0.
\]
Denoting by
\[
\xi_\alpha^\delta \in \partial\mathcal{R}(u_\alpha^\delta)
\]
the corresponding subgradient of $\mathcal{R}$, it follows that
\[
-\alpha\xi_\alpha^\delta = F^*(Fu_\alpha^\delta-v^\delta).
\]
Or, we can write
\begin{equation}\label{eq:KKT}
\xi_\alpha^\delta = F^* \omega_\alpha^\delta
\qquad\text{ with }\qquad
-\alpha\omega_\alpha^\delta = Fu_\alpha^\delta - v^\delta.
\end{equation}
As a consequence, we have
\begin{equation}\label{eq:h3_a}
\begin{aligned}
\D_{\xi_\alpha^\delta,\xi^\dagger}^{{\rm sym}}(u_\alpha^\delta,u^\dagger)
&= \langle \xi_\alpha^\delta-\xi^\dagger,u_\alpha^\delta-u^\dagger\rangle\\
&= \langle F^*\omega_\alpha^\delta - F^*\omega^\dagger,u_\alpha^\delta-u^\dagger\rangle\\
&= \langle \omega_\alpha^\delta - \omega^\dagger, Fu_\alpha^\delta - Fu^\dagger\rangle\\
&= \langle \omega_\alpha^\delta-\omega^\dagger, Fu_\alpha^\delta - v^\delta \rangle
+ \langle \omega_\alpha^\delta - \omega^\dagger, v^\delta - v^\dagger\rangle \\
&= -\alpha\langle\omega_\alpha^\delta - \omega^\dagger, \omega_\alpha^\delta \rangle
+ \langle \omega_\alpha^\delta - \omega^\dagger,v^\delta - v^\dagger\rangle\\
&= - \alpha \lVert \omega_\alpha^\delta - \omega^\dagger \rVert^2 
- \alpha \langle \omega_\alpha^\delta - \omega^\dagger,\omega^\dagger\rangle
+ \langle \omega_\alpha^\delta - \omega^\dagger,v^\delta - v^\dagger\rangle\\
&\le  - \alpha\lVert \omega_\alpha^\delta - \omega^\dagger\rVert^2
- \alpha \langle \omega_\alpha^\delta - \omega^\dagger,\omega^\dagger\rangle
+ \delta\lVert \omega_\alpha^\delta-\omega^\dagger\rVert.
\end{aligned}
\end{equation}
We next use the interpolation inequality and the definitions of $\omega_\alpha^\delta$
and $\omega^\dagger$ and obtain
\begin{equation}\label{eq:h3_b}
\begin{aligned}
-\langle \omega_\alpha^\delta - \omega^\dagger,\omega^\dagger\rangle
&= -\langle \omega_\alpha^\delta - \omega^\dagger,(FF^*)^\mu \tilde{\eta}^\dagger\rangle\\
&\le \lVert \tilde{\eta}^\dagger \rVert 
\lVert F^*(\omega_\alpha^\delta-\omega^\dagger)\rVert^{2\mu} 
\lVert \omega_\alpha^\delta-\omega^\dagger\rVert^{1-2\mu}\\
&= \lVert \tilde{\eta}^\dagger \rVert \lVert \xi_\alpha^\delta - \xi^\dagger \rVert^{2\mu}
\lVert \omega_\alpha^\delta - \omega^\dagger\rVert^{1-2\mu}.
\end{aligned}
\end{equation}
Now we can use the local $q$-coconvexity of $\mathcal{R}$ and the boundedness
of $u_\alpha^\delta$ (see Remark~\ref{re:bd}) to estimate
\[
\lVert \xi_\alpha^\delta - \xi^\dagger \rVert
\le C\D_{\xi_\alpha^\delta,\xi^\dagger}^{{\rm sym}}(u_\alpha^\delta,u^\dagger)^{1/q}
\]
and obtain from~\eqref{eq:h3_a} and~\eqref{eq:h3_b} the bound
\begin{multline}\label{eq:h3_c}
\D_{\xi_\alpha^\delta,\xi^\dagger}^{{\rm sym}}(u_\alpha^\delta,u^\dagger)\\
\le C\alpha \lVert \tilde{\eta}^\dagger\rVert \D_{\xi_\alpha^\delta,\xi^\dagger}^{{\rm sym}}(u_\alpha^\delta,u^\dagger)^{\frac{2\mu}{q}}
\lVert \omega_\alpha^\delta - \omega^\dagger\rVert^{1-2\mu}
+  \delta \lVert \omega_\alpha^\delta-\omega^\dagger\rVert- \alpha\lVert \omega_\alpha^\delta - \omega^\dagger\rVert^2.
\end{multline}

In the following, we will only treat the more difficult case $\mu < 1/2$.
For $\mu = 1/2$, the argumentation is similar but simpler, due to the absence of the
term $\lVert \omega_\alpha^\delta - \omega^\dagger\rVert^{1-2\mu}$ in the first
product on the right hand side of~\eqref{eq:h3_c}.

We use first the inequality
\[
\delta \lVert \omega_\alpha^\delta - \omega^\dagger \rVert
\le \frac{\delta^2}{2\alpha} + \frac{\alpha}{2}\lVert \omega_\alpha^\delta - \omega^\dagger\rVert^2
\]
and then the three term Young inequality~\eqref{eq:Y3} with
\[
\begin{aligned}
a &= C(1-2\mu)^{\frac{1-2\mu}{2}}\lVert \tilde{\eta}^\dagger \rVert\alpha^{\frac{1+2\mu}{2}}, &&& r &= \frac{2q}{q+2\mu q - 4\mu},\\
b &= \frac{\alpha^{\frac{1-2\mu}{2}}}{(1-2\mu)^{\frac{1-2\mu}{2}}}\lVert \omega_\alpha^\delta - \omega^\dagger\rVert^{1-2\mu} , &&& s &= \frac{2}{1-2\mu} ,\\
c &= \D_{\xi_\alpha^\delta,\xi^\dagger}^{{\rm sym}}(u_\alpha^\delta,u^\dagger)^{\frac{2\mu}{q}}, &&& t &= \frac{q}{2\mu}.
\end{aligned}
\]
Then we obtain from~\eqref{eq:h3_c} that
\[
\D_{\xi_\alpha^\delta,\xi^\dagger}^{{\rm sym}}(u_\alpha^\delta,u^\dagger)
\le C_1\frac{\delta^2}{\alpha} + C_2 \alpha^{\frac{q+2\mu q}{q+2\mu q - 4\mu}}.
\]
Again, balancing the two terms on the right hand side leads to a parameter choice
\[
\alpha \sim \delta^{\frac{q+2\mu q - 4\mu}{q+2\mu q - 2\mu}}
\]
and a convergence rate
\[
\D_{\xi_\alpha^\delta,\xi^\dagger}^{{\rm sym}}(u_\alpha^\delta,u^\dagger)
\lesssim \delta^{\frac{q+2\mu q}{q+2\mu q - 2\mu}}.
\]
Replacing again $\mu$ by $\nu - \frac{1}{2}$, we obtain the results claimed
in the statement of the theorem.
\end{proof}

\begin{remark}
The equations~\eqref{eq:KKT} are just the KKT conditions for
the optimisation problem $\min_u \mathcal{T}_\alpha(u,v^\delta)$,
and $\omega_\alpha^\delta$ can be just seen as the dual solution
of this problem.
See also~\cite{Gra13}, where the connection to a dual Tikhonov functional
is discussed.
\end{remark}

\begin{remark}
In the case where the regularisation term $\mathcal{R}$ is $2$-coconvex,
the parameter choice and convergence rate simplify to
\[
\D_{\xi_\alpha^\delta,\xi^\dagger}^{{\rm sym}}(u_\alpha^\delta,u^\dagger)
\lesssim \delta^{\frac{4\nu}{1+2\nu}}
\qquad\text{ for }\qquad
\alpha \sim \delta^{\frac{2}{1+2\nu}}.
\]
In the case of quadratic Tikhonov regularisation, these rates turn
out to be identical to the classical rates.
Indeed, the quadratic norm is obviously $2$-coconvex, since we have
for
\[
\mathcal{R}(u) = \frac{1}{2}\lVert u \rVert^2
\]
that
\[
\partial\mathcal{R}(u) = \{u\}
\qquad\text{ and }\qquad
\D_{u_1,u_2}^{{\rm sym}}(u_1,u_2)
= \lVert u_1 - u_2 \rVert^2.
\]
Moreover, the source condition simply reads as
\[
u^\dagger = (F^*F)^\nu \eta^\dagger.
\]
Together with Remark~\ref{re:quad_low2}, which deals with the lower order case, we thus recover the 
classical result that the source condition
\[
u^\dagger \in \range(F^*F)^\nu
\qquad\text{ for some }\qquad 0 < \nu \le 1
\]
implies the convergence rate
\[
\lVert u_\alpha^\delta - u^\dagger \rVert \lesssim \delta^{\frac{2\nu}{2\nu+1}}
\qquad\text{ with }\qquad \alpha \sim \delta^{\frac{2}{2\nu+1}}
\]
for quadratic Tikhonov regularisation.
\end{remark}

\section{Examples}\label{se:examples}

We now study the implications for four different non-quadratic regularisation terms,
all with different convexity properties.

\subsection{$\ell^p$-regularisation}

We consider first the case where $U = \ell^2(I)$ for some countable index set $I$,
and
\[
\mathcal{R}(u) = \frac{1}{p} \lVert u \rVert_{\ell^p}^p = \frac{1}{p}\sum_{i\in I} \lvert u_i \rvert^p
\]
for some $1 < p < 2$.
Because of the embedding $\ell^p \to \ell^2$ for $p < 2$, this term
is coercive and thus Tikhonov regularisation is well-posed.
Also, this regularisation term is $2$-convex and its conjugate
\[
\mathcal{R}^*(\xi) = \frac{1}{p_*} \lVert \xi \rVert_{\ell^{p_*}}^{p_*}
\]
is $p_*$-convex with $p_* = p/(p-1)$ being the H\"older conjugate of $p$,
implying that $\mathcal{R}$ is $p_*$-coconvex (see~\cite{Bon08} for all of these results).
Moreover,
\[
\partial \mathcal{R}(u) = \bigl( u_i \lvert u_i \rvert^{p-2} \bigr)_{i \in I}
\]
whenever $u \in \domain \partial\mathcal{R} = \ell^{p_*}$.

Thus the preceding results imply that a source condition
\[
\xi^\dagger = \bigl( u_i^\dagger \lvert u_i^\dagger \rvert^{p-2} \bigr)_{i \in I}
\in \range (F^*F)^\nu
\]
leads to a convergence rate
\[
\lVert u_\alpha^\delta - u^\dagger \rVert \lesssim \delta^{\frac{2\nu}{1+2\nu}}
\quad\text{ with }\quad
\alpha \sim \delta^{\frac{2}{1+2\nu}}
\qquad\text{ if } 0 < \nu \le \frac{1}{2},
\]
and
\[
\lVert u_\alpha^\delta - u^\dagger \rVert \lesssim \delta^{\frac{p\nu}{p-1+2\nu}}
\quad\text{ with }\quad
\alpha \sim \delta^{\frac{2p-2-2\nu p+4\nu}{p-1+2\nu}}
\qquad\text{ if } \frac{1}{2} < \nu \le 1.
\]

\subsection{$L^p$-regularisation}

Next we study the situation where $\Omega$ is some bounded
domain, $U = L^2(\Omega)$, and
\[
\mathcal{R}(u) = \frac{1}{p}\int_\Omega \lvert u(x) \rvert^p\,dx = \frac{1}{p}\lVert u \rVert_{L^p}^p
\]
for some $2 < p < +\infty$. Here we are in the opposite situation to $\ell^p$-regularisation
in that the exponent has to be larger than $2$ for the regularisation method
to be well-posed.

In this case the regularisation term itself is $p$-convex,
but its conjugate
\[
\mathcal{R}^*(u) = \frac{1}{p_*}\lVert u \rVert_{L^{p_*}}^{p_*}
\]
is $2$-convex (see again~\cite{Bon08}).
Also, we have again the representation of the subgradient of $\mathcal{R}$
as
\[
\partial \mathcal{R}(u) = u \lvert u \rvert^{p-2}
\]
whenever $u \in \domain\partial\mathcal{R} = L^{p_*}$.

As a consequence, due to the $p$-convexity and $2$-coconvexity of the 
regularisation term, the results above imply that the source
condition
\[
\xi^\dagger := u\lvert u \rvert^{p-2} \in \range (F^*F)^\nu
\]
results in the convergence rates
\[
\lVert u_\alpha^\delta - u^\dagger \rVert \lesssim \delta^{\frac{2\nu}{p-1+2\nu}}
\quad\text{ with }\quad
\alpha \sim \delta^{\frac{2p-2-2p\nu + 4\nu}{p-1+\nu}}
\qquad\text{ if } 0 < \nu \le \frac{1}{2},
\]
and
\[
\lVert u_\alpha^\delta - u^\dagger \rVert \lesssim \delta^{\frac{4\nu}{p+2\nu p}}
\quad\text{ with }\quad
\alpha \sim \delta^{\frac{2}{1+2\nu}}
\qquad\text{ if } \frac{1}{2} < \nu \le 1.
\]

\subsection{Total variation regularisation}

The next example we consider is total variation regularisation
with $U = L^2(\Omega)$, $\Omega \subset \R^2$ bounded with Lipschitz boundary,
and
\[
\mathcal{R}(u) = \lvert Du \rvert (\Omega).
\]
As discussed in Remark~\ref{re:coercive}, we have to assume in this case in addition
that constant functions are not contained in the kernel of $F$
in order for the regularisation method to be well-posed.

In the case of total variation regularisation, the regularisation term is
not strictly convex, which implies that the Bregman distance
$\D_\xi(\tilde{u},u)$ may be zero for $\tilde{u} \neq u$.
As a consequence, we cannot bound the Bregman distance from
below by any power of the norm, and therefore the total variation is 
not $p$-convex for any $p$.
On the other hand, the subdifferentials of $\mathcal{R}$ are in general not
single-valued, which implies that the total variation is neither $q$-coconvex
for any $q$.
We thus end up with only the basic results
\[
\D_{\xi^\dagger}(u_\alpha^\delta,u^\dagger) \lesssim \delta^{2\nu}
\quad\text{ with }\quad
\alpha \sim \delta^{2-2\nu}
\]
for a source condition
\[
\xi^\dagger \in \partial\mathcal{R}(u^\dagger) \cap \range(F^*F)^\nu
\qquad\text{ with }\qquad 0 < \nu \le \frac{1}{2}.
\]

\subsection{Huber regularisation}

As final example, we get back to the case $U = \ell^2(I)$ for some
countable index set $I$, but consider now the Huber regularisation term
\[
\mathcal{R}(u) = \sum_{i\in I} \phi(u_i)
\]
with
\[
\phi(t) = \begin{cases}
\frac{1}{2}t^2 & \text{ if } \lvert t \rvert \le 1,\\
\lvert t \rvert - \frac{1}{2} & \text{ if } \lvert t \rvert \ge 1.
\end{cases}
\]
Because $\phi$ is not strictly convex, neither is $\mathcal{R}$,
and thus $\mathcal{R}$ is not $p$-convex for any $p$.
However,
\[
\mathcal{R}^*(\xi) = \sum_{i\in I} \phi^*(\xi_i)
\]
with
\[
\phi^*(\zeta) = \begin{cases}
\frac{1}{2}\zeta^2 & \text{ if } \lvert \zeta \rvert \le 1,\\
+ \infty &  \text{ if } \lvert \zeta > 1,
\end{cases}
\]
which is obviously $2$-convex.
Thus the Huber regularisation term is not $p$-convex for any $p$,
but is $2$-coconvex.

Moreover, we have that
\[
\partial\mathcal{R}(u) = \bigl(\rho(u_i)\bigr)_{i\in I}
\]
with
\[
\rho(t) = \begin{cases}
1 & \text{ if } t \ge 1,\\
t & \text{ if } \lvert t \rvert \le 1,\\
-1 &\text{ if } t \le -1.\\
\end{cases}
\]

Thus we obtain the convergence rates
\[
\D_{\xi^\dagger}(u_\alpha^\delta,u^\dagger) \lesssim \delta^{2\nu}
\quad\text{ with }\quad
\alpha \sim \delta^{2-2\nu}
\qquad\text{ if } 0 < \nu \le \frac{1}{2}
\]
and
\[
\D_{\xi_\alpha^\delta,\xi^\dagger}^{{\rm sym}}(u_\alpha^\delta,u^\dagger)
\lesssim \delta^{\frac{2\nu}{1+2\nu}}
\quad\text{ with }\quad
\alpha \sim \delta^{\frac{2}{1+2\nu}}
\qquad\text{ if } \frac{1}{2} \le \nu \le 1,
\]
provided that a source condition
\[
\bigl(\rho(u_i^\dagger)\bigr)_{i\in I} \in \range (F^*F)^\nu
\]
is satisfied.

\section{Conclusion}

In this paper we have studied the implications of classical
source conditions of power type to accuracy estimates and
convergence rates for non-quadratic Tikhonov regularisation.
We have seen that very basic results can be easily obtained
without any additional conditions concerning, for instance, strong convexity
or smoothness of the regularisation term. However, these results
are not optimal in cases where such additional conditions hold,
and they also fail to reproduce the classical results for quadratic
regularisation methods.

In order to be able to obtain stronger results, we considered the situation
where either the regularisation term or its convex conjugate is
$p$-convex. In these cases, it is possible to obtain sharper estimates
in the low regularity and high regularity regions, respectively.
Also, these improved results match those classically obtained
for quadratic regularisation, although the approach we have
followed here differs significantly from the classical ones.

Still, quite a few questions remain open.
First, all the results we have discussed here were obtained only for the
case of linear inverse problems.
It seems reasonable, though, to expect that a refinement of the approach chosen in this paper
might lead to convergence rates for non-linear problems as well.
This would be particularly desirable for the case of enhanced convergence rates
in the high regularity region, where up to now no easily interpretable results are available.

Next, it is well known (see~\cite{GraHalSch08,Gra10}) that sparsity assumptions
lead to improved convergence rates of, for instance, order $\delta^{1/p}$
in the case of $\ell^p$-regularisation with $1 < p < 2$.
Therefore, it would make sense to investigate whether sparsity might in general
alter and improve error estimates and convergence rates in the case of H\"older type
fractional source conditions.
For the setting of $\ell^1$-regularisation, it is known that lower order fractional source conditions
$\partial\mathcal{R}(u^\dagger) \cap \range(F^*F)^\nu \neq \emptyset$ for any $0 < \nu \le 1/2$
imply linear convergence rates in the presence of sparsity (see~\cite{FriGra12});
for $\ell^p$-regularisation, the effect of such source conditions is still an open problem.

Finally, all these results apply strictly to Hilbert spaces only, as they
make use of fractional powers of the operator $F$ and of an interpolation inequality.
Since non-quadratic regularisation methods become more important in settings
without a Hilbert space structure, a generalisation of such source conditions
together with corresponding convergence rates to Banach spaces would be desirable.
All of these points will be subject of further investigation in the future.

\end{document}